\documentclass[12pt]{article}   
\newtheorem{theorem}{Theorem} [section]
\newtheorem{corollary}[theorem]{Corollary}
\newtheorem{conjecture}[theorem]{Conjecture}

\newtheorem{lemma}[theorem]{Lemma}
\newtheorem{example}[theorem]{Example}

\newenvironment {proof} {{\it Proof.}}{\hspace*{\fill}$\Box$\par\vspace{4mm}}
\usepackage{amssymb}
\usepackage{amsmath}
\usepackage{amsfonts}
\usepackage{graphicx,psfrag,tikz}
\usepackage{graphicx,psfrag}
\usepackage{color}
\usepackage{array}
\usepackage{colortbl}
\usepackage{tikz}
\begin{document}

\bibliographystyle{plain}

\title{A Generalization of Alternating Sign Matrices}

 \author{
 Richard A. Brualdi\\
 Department of Mathematics\\
 University of Wisconsin\\
 Madison, WI 53706\\
 {\tt brualdi@math.wisc.edu}  
 \and
 Hwa Kyung Kim\footnote{Research performed while on leave as an Honorary Fellow at the University of Wisconsin-Madison.}
 \\
 Department of Mathematics Education\\
 Sangmyung University\\
 Seoul 110-743, South Korea\\
 {\tt indices@smu.ac.kr}}

\maketitle

 \begin{abstract}
 In alternating sign matrices the first and last nonzero entry in each row and column is specified to be $+1$.
 Such matrices always exist. We investigate a generalization   by specifying independently the sign of the first and last nonzero entry in each row and column to be either a $+1$ or a  $-1$. We determine necessary and sufficient conditions for such matrices to exist.

 
\end{abstract}

\section{Introduction}
An {\it alternating sign matrix}, abbreviated ASM, is an $n\times n$ $(0,+1,-1)$-matrix such that, ignoring $0$s, the $+1$s and $-1$s in each row and column alternate, beginning and ending with a $+1$.  ASMs exist for all $n$, since permutation matrices are ASMs; in fact, the number of  $n\times n$ ASMs, as conjectured by Mills, Robbins, and Rumsey \cite{MRR} (see also \cite{DPR})  and proved independently by Kuperberg \cite{GK} and Zeilberger \cite{DZ}, equals
\[\frac{1!4!7!\cdots (3n-2)!}{n!(n+1)!(n+2)!\cdots (2n-1)!}.\]
 The book \cite{DB} by Bressoud  describes 
the history of ASMs and its connections   to partitions, tilings, and statistical physics.

By definition the first and last nonzero entry in each row and column of an ASM must be a $+1$.
In this note we offer a generalization of ASMs by allowing the first and last nonzero entry in each row and column to be arbitrarily prescribed. In this generality, matrices need not be square. In addition, existence is not automatic,  and we determine necessary and sufficient conditions for existence.

We consider $(m+2)\times (n+2)$ matrices in which the rows are indexed by $0,1,2,\ldots,m,m+1$ and the columns are indexed by $0,1,2,\ldots,n,n+1$.
Let $u=(u_1,u_2,\ldots,u_n), u'=(u_1',u_2',\ldots,u_n'), v=(v_1,v_2,\ldots,v_m)$, and $v'=(v_1',v_2',\ldots,v_m')$ 
be vectors of $\pm1$s.
A $(u,u'|v,v')$-ASM is an $m\times n$ $(0,\pm 1)$-matrix  $A$ (with rows indexed by $1,2,\ldots,m$ and columns indexed by $1,2,\ldots,n$) such that the $+1$s and $-1$s  in  rows  $1,2,\ldots,m$ and columns $1,2,\ldots,n$ of the $(m+2)\times (n+2)$
$(0,\pm 1)$-matrix $A'$ below alternate:
\begin{equation}\label{eq:asm}
A'=\begin{array}{c||ccccc||c}
0&u_1&u_2&\cdots &u_{n-1}&u_n&0\\ \hline\hline
v_1&&&&&&v_1'\\
v_2&&&&&&v_2'\\
\vdots&&&A&&&\vdots\\
v_{m-1}&&&&&&v_{m-1}'\\ 
v_m&&&&&&v_m'\\ \hline\hline
0&u_1'&u_2'&\cdots&u_{n-1}'&u_n'&0\end{array}.\end{equation}
We write $A=A'[1,2,\ldots,m|1,2,\ldots,n]$.

\bigskip\noindent 
{\bf Examples:} 
\begin{itemize}
\item[\rm (i)] Let $u=(+1,-1,-1,+1)$, $u'=(+1,-1,+1,-1)$, $v=(+1,-1,+1,-1)$, and $v'=(-1,+1,+1,-1)$. Then, partitioning rows and columns and omitting $0$s and  using only the signs of the nonzero entries. the matrix
\[A'=\begin{array}{c||c|c|c|c||c}
&+&-&-&+&\\ \hline\hline
+&-&&+&&-\\ \hline
-&&+&&-&+\\ \hline
+&&&-&&+\\ \hline
-&&&&+&-\\ \hline\hline
&+&-&+&-&\end{array}\]
gives the $(u,u'|v,v')$-ASM
\[A=\left[\begin{array}{c|c|c|c}
-&&+&\\ \hline
&+&&-\\ \hline
&&-&\\ \hline
&&&+\end{array}\right].\]
\item[\rm (ii)] Let $u=(+1,-1,-1,-1)$, $u'=(+1,-1,-1,+1)$, $v=(+1,-1,+1)$, and $v'=(-1,-1,-1)$, then the matrix
 \[A'=\begin{array}{c||c|c|c|c||c}
&+&-&-&-&\\ \hline\hline
+&&&&&-\\ \hline
-&&+&&&-\\ \hline
+&-&&+&&-\\  \hline\hline
&+&-&-&+&\end{array}\]
gives the $(u,u'|v,v')$-ASM
\[A=\left[\begin{array}{c|c|c|c}
\phantom{+}&&&\phantom{+}\\ \hline
&+&&\\ \hline
-&&+&\end{array}\right].\]
\item[\rm (iii)] If $u_iu_i'=-1$ for all $i=1,2,\ldots,n$ and $v_jv_j'=-1$ for all $j=1,2,\ldots,m$, then the $m\times n$ zero matrix $O_{m,n}$ is a 
$(u,u'|v,v')$-ASM.
\item[\rm (iv)] Let $u=(+1,+1)$, $u'=(-1,+1)$, $v=(+1,+1)$, $v'=(+1,-1)$. Then a $(u,u'|v,v')$-ASM does not exist since it is impossible to complete
the matrix
\[\begin{array}{c||c|c||c}
&+&+&\\ \hline\hline
+&&&+\\ \hline
-&&&-\\ \hline\hline
&-&+&\end{array}\]
to have alternating signs.
\item[\rm (v)] If $u=u'=v=v'=(-1,-1,\ldots,-1)$, then a  $(u,u'|v,v')$-ASM is an ordinary ASM. Note that if $u=u'$ and $v=v'$ are, respectively,  $n$-vectors and $m$-vectors of all $-1$s, then if $m\ne n$,  a  $(u,u'|v,v')$-ASM does not exist.
\item[\rm (vi)] If $A$ is any $(0,\pm 1)$-matrix in which the $+1$s and $-1$s alternate in each row and column, then we may choose $u,u',v,v'$ so that $A$ is a  $(u,u'|v,v')$-ASM: for each $i$, choose $u_i$ and $u_i'$ to have the opposite sign of the, respectively, first and last nonzero entry in column $i$ of $A$ and, for each $j$, choose $v_j$ and $v_j'$ to have the opposite sign of the, respectively, first and last nonzero entry in row $j$ of $A$. If row $i$ of $A$ has only zeros, then we may choose $u_i$ equal to $+1$ or $-1$, and then choose $u_i'=-u_i$. A similar choice is made when a column contains only zeros.  
\end{itemize}

For a vector $w=(w_1,w_2,\ldots,w_p)$, let $w^{\leftarrow}=(w_p,\ldots,w_2,w_1)$ be the vector obtained from $w$ by reversing its coordinates.
If $A$ is a matrix, then by 
$A_{\uparrow}$ we mean the matrix obtained from $A$ by inverting the order of its rows, and by  $A^{\leftarrow}$ we mean the matrix obtained from $A$ by inverting the order of its columns. By $A_{\uparrow}^{\leftarrow}$ we mean the matrix obtained from $A$ by inverting the order of both its rows and columns.
Let $A$ be a $(u,u'|v,v')$-ASM. Then the matrix $A_{\uparrow}$ is a $(u',u|v^{\leftarrow},v'^{\leftarrow})$-ASM, and the matrix $A^{\leftarrow}$ is a $(u^{\leftarrow},u'^{\leftarrow}|v',v)$-ASM. 
The matrix $A_{\uparrow}^{\leftarrow}$  is a  $(u'^{\leftarrow},u^{\leftarrow}|v'^{\leftarrow},v^{\leftarrow})$-ASM. Also the transpose matrix 
$A^T$ is a $(v,v'|u,u')$-ASM. 
  
In the next section we obtain necessary and sufficient conditions for the existence of $(u,u'|v,v')$-ASMs.

\section{Main Theorem}

We first obtain some necessary conditions for a $(u,u'|v,v')$-ASM to exist. For this we need to define certain quantities.

Let $u=(u_1,u_2,\ldots,u_n), u'=(u_1',u_2',\ldots,u_n'), v=(v_1,v_2,\ldots,v_m)$, and $v'=(v_1',v_2',\ldots,v_m')$ 
be vectors of $\pm1$s. We define $m$-vectors $r^+(v,v')=(r^+_1,r^+_2,\ldots,r^+_m)$ and  $r^-(v,v')=(r^-_1,r^-_2,\ldots,r^-_m)$,  and $n$-vectors  $c^+(u,u')=(c^+_1,c^+_2,\ldots,c^+_n)$ and $c^-(u,u')=(c^-_1,c^-_2,\ldots,c^-_n)$ as follows:
\begin{enumerate}
\item[]  $r^+_k=r^+_k(v,v')$ is the number of $i\le k$ such that $v_i=v_i'=+1$. 
\item[]  $r^-_k=r^-_k(v,v')$ is the number of $i\le k$ such that $v_i=v_i'=-1$. 
\item[]  $c^+_l=c^+_l(u,u')$ is the number of $j\le l$ such that $u_j=u_j'=+1$. 
\item[]  $c^-_l=c^-_l(u,u')$ is the number of $j\le l$ such that $u_j=u_j'=-1$. 
\end{enumerate}
In particular, we have:
\begin{enumerate}
\item[] $r^+_m=r^+_m(v,v')$ is the total number of $i$ such that $v_i=v_i'=+1$.
\item[]  $r^-_m=r^-_m(v,v')$ is the total number of $i$ such that $v_i=v_i'=-1$.
\item[]  $c^+_n=c^+_n(u,u')$ is the total number of $j$ such that $u_j=u_j'=+1$.
 \item[]  $c^-_n=c^-_n(u,u')$ is the total number of $j$ such that $u_j=u_j'=-1$.
\end{enumerate}
 
Now let  $A=[a_{ij}]$ be a $(u,u'|v,v')$-ASM as determined by (\ref{eq:asm}).  For each $j$ such that $u_j=u_j'=+1$,   column $j$ of $A$ contains  one more $-1$ than $+1$, and for each $j$ such that $u_j=u_j'=-1$, column $j$ of $A$ contains one more
$+1$ than $-1$. If $u_j$ and $u_j'$ are of opposite sign, then column $j$ contains an equal number of $+1$s and $-1$s.
Hence the sum of the entries of $A$ equals $c^-_n(u,u')-c^+_n(u,u')$. Arguing by rows rather than columns, we see that the sum of the entries of $A$ also equals $r^-_m(v,v')-r^+_m(v,v')$. Hence 
\begin{equation}\label{eq:nec1}
r^-_m(v,v')-r^+_m(v,v')=c^-_n(u,u')-c^+_n(u,u').\end{equation}
We set
\begin{enumerate}
\item[] $u^+=|\{j: u_j=+1\}|$,
\item[] $u^-=|\{j: u_j=-1\}|$,
\item[] $v^+=|\{i: v_i=+1\}|$,
\item[] $v^-=|\{i: v_i=-1\}|$.
\end{enumerate}
Now consider the leading $k\times n$ submatrix $A[1,2,\ldots,k|1,2,\ldots,n]$ of the 
$(u,u'|v,v')$-ASM $A=[a_{ij}]$. Then
\[
r^-_k-r^+_k=r_k^-(v,v') - r_k^+(v,v')= \sum_{i=1}^k\sum_{j=1}^n a_{ij}=\sum_{j=1}^n\left(\sum_{i=1}^ka_{ij}\right).\]
We have
\[\sum_{i=1}^ka_{ij}=\left\{\begin{array}{ll}
0\mbox{ or }-1&\mbox{ if $u_j=+1$},\\
0\mbox{ or } +1 &\mbox{ if $u_j=-1$.}\end{array}\right.\]
Therefore, the maximum 
$\sum_{j=1}^n\left(\sum_{i=1}^ka_{ij}\right)$ can be is $u^-$ and the minimum it can be is
$-u^+$. Hence we have
\begin{equation}\label{eq:nec2}
-u^+\le r_k^-(v,v') - r_k^+(v,v')\le u^-\quad (k=1,2,\ldots,m).\end{equation}
In a similar way, we obtain
\begin{equation}\label{eq:nec3}
-v^+\le c_l^-(u,u') - c_l^+(u,u')\le v^-\quad (l=1,2,\ldots,n).\end{equation}
In summary, (\ref{eq:nec1}), (\ref{eq:nec2}), and (\ref{eq:nec3}) are necessary conditions for the existence of a $(u,u'|v,v')$-ASM. Similar conditions must hold by arguing with $A_{\uparrow}$ and $A^{\leftarrow}$ but it turns out that they are not needed (that is, are consequences of 
(\ref{eq:nec1}), (\ref{eq:nec2}), and (\ref{eq:nec3})) .

\begin{theorem}\label{th:one}
Let $u=(u_1,u_2,\ldots,u_n), u'=(u_1',u_2',\ldots,u_n'), v=(v_1,v_2,\ldots,v_m)$, and $v'=(v_1',v_2',\ldots,v_m')$ 
be vectors of $\pm1$s. Then a $(u,u'|v,v')$-ASM exists if and only if 
 $(\ref{eq:nec1})$, $(\ref{eq:nec2})$, and $(\ref{eq:nec3})$  hold.
\end{theorem}

\begin{proof}
We have already established that a $(u,u'|v,v')$-ASM implies that  (\ref{eq:nec1}), (\ref{eq:nec2}), and (\ref{eq:nec3}) hold. Now assume that  (\ref{eq:nec1}), (\ref{eq:nec2}), and (\ref{eq:nec3}) are satisfied. We prove by induction on $m+n$ that there exists a $(u,u'|v,v')$-ASM $A=[a_{ij}]$. If $m=n=1$, then the conclusion holds. Assume that $m+n\ge 3$.

Without loss of generality we assume that $r^+_m(v,v')\le c^+_n(u,u')$. By (\ref{eq:nec1}), this implies  that $r^-_m(v,v')\le c^-_n(u,u')$.

\smallskip\noindent
{\it Case $1$}: $r^+_m(v,v')\ge 1$.
Let $q$ be an index with $v_q=v_q'=+1$. We determine the {\it longest} sequence $1\le s_0<t_1<s_1<t_2<\cdots <t_p<s_p$, with $0\le p\le c^+_n-r^+_m$, such that
$u_{s_i}=u'_{s_i}=+1$ for $i=0,1,\ldots,p$ and $u_{t_j}=u'_{t_j}=-1$ for $j=1,2,\ldots,p$; here the integers $s_0,t_1,s_1,t_2,\ldots,t_p,s_p$ are chosen, in turn,  to be the first available satisfying the specified conditions (later we refer to this as the {\it first available property}).
We then set $a_{q,s_i}=-1$ for $i=0,1,\ldots,p$ and $a_{q,t_j}=+1$ for $j=1,2,\ldots,p$. 
All other entries in row $q$ are set equal to 0. We delete row $q$ and columns $s_0,t_1,s_1,t_2,\ldots,t_p, s_p$. We also delete $u_{s_0},u_{t_1},u_{s_1},u_{t_2},\ldots,u_{t_p},u_{s_p}$ from $u$, 
delete $u'_{s_0},u'_{t_1},u'_{s_1},u'_{t_2},\ldots,u'_{t_p},u'_{s_p}$ from $u'$, delete $v_q$ from $v$, and delete  $v'_q$ from $v'$ to get, respectively, vectors $\bar{u}$ and $\bar{u}'$ of length $\bar{n}=n-(2p+1)$, and vectors $\bar{v}$ and $\bar{v}'$ of length $\bar{m}=m-1$.  For rows, columns, and vector entries, we use the same indices that they had before the deletions.
We now show that $\bar{u},\bar{u}', \bar{v}$, and $\bar{v}'$  satisfy the corresponding conditions  (\ref{eq:nec1}), (\ref{eq:nec2}), and (\ref{eq:nec3}), where we use the corresponding notations: 
$\bar{r}^-_{\bar{m}}, \bar{r}^+_{\bar{m}}, \bar{c}^-_{\bar{n}}$, and $\bar{c}^+_{\bar{n}}$.

We have $\bar{u}^+=u^+-p-1$, $\bar{u}^-=u^- - p$, $\bar{v}^+=v^+-1$, and $\bar{v}^-=v^-$.  We also have
$\bar{r}^+_{\bar{m}}=r^+_m-1$, $\bar{r}^-_{\bar{m}}=r^-_m$, $\bar{c}^+_{\bar{n}}=c^+_n-p-1$, and $\bar{c}^-_{\bar{n}}=c^-_n-p$. Since $p\le c^+_n-r^+_m=c_n^--r_m^-$, we  have
\[\bar{c}^+_{\bar{n}}=c^+_n-p-1\ge r_m^+-1= \bar{r}^+_{\bar m}\mbox{ and } \bar{c}^-_{\bar{n}}\ge \bar{r}^-_{\bar{m}}.\]
We now check that the corresponding conditions (\ref{eq:nec1}), (\ref{eq:nec2}), and (\ref{eq:nec3}) hold for our new parameters. Using the above values, we see that
\[\bar{r}^-_{\bar{m}}-\bar{r}^+_{\bar{m}}=\bar{c}^-_{\bar{n}}-\bar{c}^+_{\bar{n}},\]
verifying the corresponding condition (\ref{eq:nec1}). For the corresponding condition (\ref{eq:nec2}), we calculate that for each $k$ we have 
\[\bar{r}^-_k-\bar{r}^+_k\le \bar{r}^-_k\le \bar{r}^-_m\le  \bar{c}^-_{\bar{n}} \le \bar{u}^-, \]
and 
\[\bar{r}^+_k-\bar{r}^-_k\le \bar{r}^+_k\le \bar{r}^+_m< \bar{c}^+_{\bar{n}} \le \bar{u}^+, \]
which together are equivalent to the corresponding condition (\ref{eq:nec2}).

We now turn to verifying the corresponding condition (\ref{eq:nec3}). For this we need to examine several possibilities according to the value of $k$. We need to show that for each $l$, 
\[\bar{c}^-_l-\bar{c}^+_l\le \bar{v}^- \mbox{ and }
\bar{c}^+_l-\bar{c}^-_l\le \bar{v}^+.\]

(a) $l<s_0$: By the first available property, we have that $\bar{c}^+_l=0$, and hence
\[\bar{c}^-_l-\bar{c}^+_l={c}^-_l-{c}^+_l\le v^-=\bar{v}^-,\]
and
\[\bar{c}^+_l-\bar{c}^-_l=-\bar{c}^-_l \le 0\le\bar{v}^+.\]

\smallskip
(b) $s_j<l<t_{j+1}$ for some $j$ with $0\le j<p$: Again using the first available property, we calculate that
\[\bar{c}^-_l-\bar{c}^+_l\le \bar{c}^-_{s_j-1} -\bar{c}^+_{s_j-1}= {c}^-_{s_j-1} -{c}^+_{s_j-1}
\le v^-=\bar{v}^-.\]
and 
\[\bar{c}^+_l-\bar{c}^-_l=c^+_l-c^-_l-1\le v^+-1=\bar{v}^+.\]

\smallskip
(c) $t_j<l<s_j$  for some $j$ with $1\le j\le p$: Using the first available property once more, we obtain that
\[\bar{c}^-_l-\bar{c}^+_l= {c}^-_{l}-{c}^+_{l}= v^-=\bar{v}^-,\]
and
\[\bar{c}^+_l-\bar{c}^-_l\le \bar{c}^+_{t_j-1}-\bar{c}^-_{t_j-1}=
 {c}^+_{t_j-1}-{c}^-_{t_j-1}-1
\le v^+-1=\bar{v}^+.\]

\smallskip
(d) $s_p<l$: We first calculate that
\[\bar{c}^+_l-\bar{c}^-_l={c}^+_l-{c}^-_l-1\le v^+-1=\bar{v}^+.\]
For $\bar{c}^-_l-\bar{c}^+_l$, we consider three subcases:
\smallskip

(d1) If there does not exist a column $i$ with $s_p<i<l$ such that $u_i=u_i'=-1$, then the first available property gives
\[\bar{c}^-_l-\bar{c}^+_l\le \bar{c}^-_{s_p-1}-\bar{c}^+_{s_p-1}
={c}^-_{s_p-1}-{c}^+_{s_p-1}\le v^-=\bar{v}^-.\]

\smallskip
(d2) If there exists a column $i$ with $s_p<i<l$ such that $u_i=u_i'=-1$, and there exists a column $j$ with $j>i$ such that $u_j=u_j'=+1$, then
by the longest property  of the sequence 
$s_0,t_1,s_1,t_2,\cdots ,t_p,s_p$  with $0\le p\le c^+_n-r^+_m$, we have
$p=c_n^+-r_m^-$. Thus
\[\bar{c}^-_l-\bar{c}^+_l\le \bar{c}^-_l\le c^-_n-(c^+_n-r^+_m)=(c^-_n-c^+_n)+r^+_m=
(r^-_m-r^+_m)+r^+_m=r^-_m=\bar{v}^-.\]

\smallskip
(d3) If there exists a column $i$ with $s_p<i<l$ such that $u_i=u_i'=-1$, and there does not exist a column $j$ with $j>i$ such that $u_j=u_j'=+1$
(so $c^+_l=c^+_n$), then, using also our assumption in this case that $r^+_m\ge 1$, we have
\[\bar{c}^-_l-\bar{c}^+_l={c}^-_l-{c}^+_l+1=
c^-_l-c^+_n+1\le c_n^{-}-c^+_n+1=r^-_m-r^+_m+1\le r^-_m =\bar{v}^-.\]

\smallskip\noindent
{\it Case $2$}: $r^-_m(v,v')\ge 1$. 
This case is very similar to Case 1 and we omit the details. We choose an index
$q$  with $v_q=v_q'=-1$. We determine the longest sequence $1\le s_0<t_1<s_1<t_2<\cdots <t_p<s_p$, with $0\le p\le c^+_n-r^+_m$, such that
$u_{s_i}=u'_{s_i}=-1$ for $i=0,1,\ldots,p$ and $u_{t_j}=u'_{t_j}=+1$ for $j=1,2,\ldots,p$; here the integers $s_0,t_1,s_1,t_2,\ldots,t_p,s_p$ are chosen, in turn,  to be the first available satisfying the specified conditions.
We then set $a_{q,s_i}=+1$ for $i=0,1,\ldots,p$ and $a_{q,t_j}=-1$ for $j=1,2,\ldots,p$. 
All other entries in row $q$ are set equal to 0. We delete row $q$ and columns $s_0,t_1,s_1,t_2,\ldots,t_p, s_p$, and proceed as in Case 1.

\smallskip\noindent
{\it Case $3$}: $r^+_m(v,v')=r_m^-(v,v')=0$. 

The verification in this case is similar to that in Case 1 but there are important differences so that we include the details.
 By (\ref{eq:nec1}), we have $c^+_n (u,u')=c_n ^- (u,u')$.  
Without loss of generality, we may assume that $u_1 = u'_1 =+1$. There is $q$ such that $v_q = +1$ and $v'_q=-1$. 
In this case  we determine the {\em longest } sequence $ 1 = s_1 < t_1 < \cdots < s_p < t_p$, such that $u_{s_i}=u'_{s_i}=+1$ for $i=1, \ldots, p$ and $u_{t_j}=u'_{t_j}=-1$ for $j=1, \ldots, p$; here the integers $s_1 , t_1 , \ldots, s_p, t_p$ are chosen, in turn, to be the first available satisfying the specified conditions.
We set $a_{q,s_i} = -1$ for $i=1, \ldots, p$ and $a_{q, t_j} = +1$ for $j=1, \ldots, p$. All other entries in row $q$ are set equal to 0. We delete row $q$ and columns $s_1, t_1, \ldots, s_p, t_p$. 
We also delete $u_{s_1}, u_{t_1}, \ldots, u_{s_p}, u_{t_p}$ from $u$, delete $u'_{s_1}, u'_{t_1}, \ldots, u'_{s_p}, u'_{t_p}$  from $u'$, delete $v_q$ from $v$, and delete $v'_q$ from $v'$ to get, respectively, vectors ${\bar u}$ and ${\bar u'}$ of length ${\bar n}=n-2p$, and vectors ${\bar v}$ and ${\bar v'}$ of length ${\bar m}=m-1$. 
We adopt the same conventions as in Case 1.
We now show that ${\bar u}, {\bar u'}, {\bar v}$, and ${\bar v'}$ satisfy the corresponding conditions (2), (3), and (4). 

\medskip

We  have ${\bar u^+}=u^+ -p$, ${\bar u^-} = u^- -p$, ${\bar v^+ } = v^+-1$, and ${\bar v^- } =v^-$.  We also have ${\bar r}_{\bar m}^+ = {\bar r}_{\bar m}^-=0$, ${\bar c}_{\bar n}^+=c_n ^+ -p$, and ${\bar c}_{\bar n}^-=c_n ^- -p$. We now check that the corresponding conditions (2), (3), and (4) hold for our new parameters. First, we have 
\[
{\bar r}_{\bar m}^- - {\bar r}_{\bar m}^+ = {\bar c}_{\bar n}^- - {\bar c}_{\bar n}^+,
\]
verifying the corresponding condition (\ref{eq:nec1}). Since ${\bar r}_{\bar m}^+ = {\bar r}_{\bar m}^-=0$, we have ${\bar r}_k^+ = {\bar r}_k^-=0$ for each $k$, and hence the corresponding  condition (\ref{eq:nec2}) holds. 

To verify the corresponding condition (4),  we need to show that for each $l$, 
\[
{\bar c_l }^- - {\bar c_l }^+ \le {\bar v}^-
\mbox{ and }
{\bar c_l }^+ - {\bar c_l }^- \le {\bar v}^+
.\]
To do this, we  consider several possibilities according to the value of $l$.

\smallskip
(a) $ s_j < l < t_j$ for some $j$ with $1 \le j \le p$: By the first available property, we calculate that
\[
{\bar c_l }^- -  {\bar c_l }^+ \le {\bar c_{s_j-1} }^- -  {\bar c_{s_j-1} }^+ = c_{s_j-1} ^- -  c_{s_j-1} ^+ \le v^- = {\bar v}^-,
\]
and
\[
{\bar c_l }^+ -  {\bar c_l }^- = c_{l} ^+ -  c_{l} ^- -1\le v^+ -1 = {\bar v}^+.
\]

(b) $ t_j < l < s_{j+1}$ for some $j$ with $1 \le j \le p-1$: Again the first available property implies that 
\[
{\bar c_l }^- -  {\bar c_l }^+ = c_{l} ^- -  c_{l} ^+ \le v^-= {\bar v}^-,
\]
and
\[
{\bar c_l }^+ -  {\bar c_l }^- \le {\bar c_{t_j-1} }^+ -  {\bar c_{t_j-1} }^- = c_{t_j-1} ^+ -  c_{t_j-1} ^- -1 \le v^+-1 = {\bar v}^+.
\]
\medskip

(c) $t_p < l$: We first calculate that 
\[
{\bar c_l }^- -  {\bar c_l }^+ = c_{l} ^- -  c_{l} ^+ \le v^- = {\bar v}^-.
\]
For ${\bar c_l }^+ -  {\bar c_l }^-$, we consider two subcases:
\medskip

(c1) If there does not exist a column $i$ with $t_p < i < l$ such that $u_i = u'_i=+1$, then 
\[
{\bar c_l }^+ -  {\bar c_l }^-  \le {\bar c_{t_p-1} }^+ -  {\bar c_{t_p-1} }^- = c_{t_p-1} ^+ -  c_{t_p-1} ^- -1 \le v^+-1 = {\bar v}^+.
\]

\medskip

(c2) If there exists a column $i$ with $t_p < i < l$ such that $u_i = u'_i=+1$, then by the longest property of the sequence $s_1, t_1, \ldots, s_p, t_p$, we have there does not exist a column $j$ with $j>i$ such that $u_j = u'_j =-1$ (so $c_l ^- = c_n ^-$). Thus
\[
{\bar c_l }^+ -  {\bar c_l }^-  = c_l^+ -  c_l^- = c_{l} ^+ -  c_{n} ^-  \le c_n^+ - c_n ^-=0 \le  {\bar v}^+.
\]

All cases having been argued the proof of the theorem is now complete.
\end{proof}


The proof Theorem \ref{th:one} gives an algorithm to construct a $(u,u'|v,v')$-ASM when the conditions (\ref{eq:nec1}),  (\ref{eq:nec2}), and  (\ref{eq:nec3})  are satisfied. We give some examples illustrating the proof.

\smallskip\noindent
{\bf Example 1.} Let $m=n$ and $u=u'=v=v'=(-1,-1,\ldots,-1)$. Then a $(u,u'|v,v')$-ASM is a classical ASM. In applying the algorithm, we have $n$ steps and we get to  choose the order in which we select the rows.
If we select them in the order $q_1,q_2,\ldots,q_n$, then the algorithm gives the permutation matrix corresponding to the permutation $(q_1,q_2,\ldots,q_n)$ of $\{1,2,\ldots,n\}$.

\smallskip\noindent
{\bf Example 2.} Let $u=u'=(+1,+1,-1,-1)$, $v=(+1,+1,+1,+1)$, and $v'=(-1,-1,-1,-1)$. Choosing rows in the order 1 and 2, we obtain
\[\begin{array}{c||c|c|c|c||c}
& +&+&-&-&\\ \hline\hline
+&-&&+&&-\\ \hline
+&&-&&+&-\\ \hline
+&&&&&-\\ \hline
+&&&&&-\\ \hline\hline
&+&+&-&-&\end{array}.\]
Note that for the alternating condition, nonzeros are not required in any of the rows but are needed to obtain the alternating property of the columns.

\smallskip\noindent
{\bf Example 3.} Let $m=5$ and $n=16$, and let \\\\
\[u=u'=(-1,-1,+1,+1,+1,-1,-1,-1,+1,+1,+1,+1,-1,-1,-1,-1)\], 
\[v=(-1,+1,-1,-1,+1) \mbox{ and }v'=(-1,-1,-1,-1,+1).\]
Choosing rows in order 1,3,5,4,2, the algorithm produces
\[\begin{array}{c||c|c|c|c|c|c|c|c|c|c|c|c|c|c|c|c||c}
&-&-&+&+&+&-&-&-&+&+&+&+&-&-&-&-&\\ \hline\hline
-&+&&-&&&+&&&-&&&&+&&&&-\\ \hline
+&&&&&&&&&&&&-&&&&+&-\\ \hline
-&&+&&-&&&+&&&-&&&&+&&&-\\ \hline
-&&&&&&&&&&&&&&&+&&-\\ \hline
+&&&&&-&&&+&&&-&&&&&&+\\ \hline\hline
&-&-&+&+&+&-&-&-&+&+&+&+&-&-&-&-&\end{array}.\]

As in the examples, the algorithm always produces a $(u,u'|v,v')$-ASM with, except for possible zero rows and columns,  one nonzero in each column or one nonzero in each row.

Finally we give two examples to show that if we either did not choose the longest sequence satisfying our properties or did not use the first available property, then the algorithm may fail.
These are, respectively, 
\[\begin{array}{c||c|c|c|c||c}
&+&-&+&-&\\ \hline\hline
-&&&&&+\\ \hline
+&-&+&0&0&-\\ \hline
-&&&&&+\\ \hline
-&&&&&-\\ \hline\hline
&+&-&+&-&\end{array}\quad \mbox{ and }\quad
\begin{array}{c||c|c|c|c||c}
&+&-&+&-&\\ \hline\hline
+&&&&&-\\ \hline
+&-&0&0&+&-\\ \hline
+&&&&&-\\ \hline
+&&&&&-\\ \hline\hline
&+&-&+&-&\end{array}.\]
Neither of these can be completed to the required ASM.

\section{Concluding Remarks and Questions}

Let $u$ and $u'$ be $n$-vectors of $\pm 1$s and let $v$ and $v'$ be $m$-vectors of $\pm 1$s. 
Let ${\mathcal A}_{m,n}(u,u'|v,v')$ be the set of all $(u,u'|v,v')$-ASMs. In case $m=n$ and $u=u'=v=v'=(-1,-1,\ldots,-1)$, we get the set of all $n\times n$ classical ASMs and we abbreviate this notation to ${\mathcal A}_n$. Let
\[f(u,u'|v,v')=|{\mathcal A}_{m,n}(u,u'|v,v')|\]
where for brevity we write
\[f(n)=|{\mathcal A}_n|=\frac{1!4!7!\cdots (3n-2)!}{n!(n+1)!(n+2)!\cdots (2n-1)!}.\]

\smallskip\noindent
{\bf Question:} If $m=n$, is $f(u,u'|v,v')\le f(n)$ for all $u,u',v,v'$?

\smallskip\noindent
This seems likely but, if true, very difficult to prove. To prove it, one would probably have to establish an injection
\[{\mathcal A}_{m,n}(u,u'|v,v')\rightarrow {\mathcal A}_n.\]

  The following  special case is ``close to'' classical ASMs.
For $0\le k\le n$, let $\alpha_{n,k}= (+1,\ldots,+1,-1,\ldots,-1)$ where there are $k$\   $+1$s followed by $(n-k)$\  $-1$s.

\begin{theorem}\label{th:two} Let $0\le k\le n$, Then
\[{\mathcal A}(\alpha_{n,k},\alpha_{n,k}|\alpha_{n,k},\alpha_{n,k})=(-{\mathcal A}_k)\oplus {\mathcal A}_{n-k},\] and
\[f(\alpha_{n,k},\alpha_{n,k}|\alpha_{n,k},\alpha_{n,k})=f(k)f(n-k)\le f(n).\]
\end{theorem}

\begin{proof} 
First, by $(-{\mathcal A}_k)\oplus {\mathcal A}_{n-k}$ we mean the set of all matrices of the form $(-A_1)\oplus A_2$, where $A_1\in {\mathcal A}_k$ and $A_2\in {\mathcal A}_{n-k}$.  Let $A\in {\mathcal A}(\alpha_{n,k},\alpha_{n,k}|\alpha_{n,k},\alpha_{n,k})$. It is easy to verify that  the $k\times (n-k)$ 
submatrix $A[1,2,\ldots,k|k+1,\ldots,n]$ and $(n-k)\times k$ submatrix $A[k+1,\ldots,n|1,2,\ldots,k]$  of $A$, respectively, below the $-1$s in $u=\alpha_{n,k}$ and to the left of the $+1$s in $v'=\alpha_{n,k}$, and above the $+1$s in $u'=\alpha_{n,k}$ and to the right of the $-1$s in $v=\alpha_{n,k}$, are zero matrices. Thus $A=(-A_1)\oplus A_2$ where $A_1\in {\mathcal A}_k$ and $A_2\in \mathcal{A}_{n-k}$. Thus
$f(\alpha_{n,k},\alpha_{n,k}|\alpha_{n,k},\alpha_{n,k})=f(k)f(n-k)$. The mapping
\[A=(-A_1)\oplus A_2\rightarrow A_1\oplus A_2\]
is an injection from ${\mathcal A}(\alpha_{n,k},\alpha_{n,k}|\alpha_{n,k},\alpha_{n,k})$ into ${\mathcal A}_n$, completing the proof.
\end{proof}

Another special case perhaps deserving  special attention is the case where $n$ is odd and the $+1$s and $-1$s alternate:  $u=u'=v=v'=(+1,-1,+1,\ldots,-1,+1)$.
In this case, there exists a $(u,u'|v,v')$-ASM with no zeros. For instance, with $n=5$, we get
\[A=\begin{array}{c||c|c|c|c|c||c}
&+&-&+&-&+&\\ \hline\hline
+&-&+&-&+&-&+\\ \hline
-&+&-&+&-&+&-\\ \hline
+&-&+&-&+&-&+\\ \hline
-&+&-&+&-&+&-\\ \hline
+&-&+&-&+&-&+\\ \hline\hline
&+&-&+&-&+&\end{array}.\]
Other $(u,u'|v,v')$-ASMs can be obtained from $A$ by replacing $2\times 2$ 
($2l\times 2l$ in general) submatrices with consecutive rows and columns with zero matrices. For example,
\[
\begin{array}{c||c|c|c|c|c||c}
&+&-&+&-&+&\\ \hline\hline
+&-&+&-&+&-&+\\ \hline
-&+&-&0&0&+&-\\ \hline
+&-&+&0&0&-&+\\ \hline
-&+&-&+&-&+&-\\ \hline
+&-&+&-&+&-&+\\ \hline\hline
&+&-&+&-&+&\end{array}.\]

 But not every $(u,u'|v,v')$-ASM can be obtained this way. For instance, 
\[\begin{array}{c||c|c|c|c|c||c}
&+&-&+&-&+&\\ \hline\hline
+&-&+&-&&&+\\ \hline
-&+&-&&+&&-\\ \hline
+&-&&+&&-&+\\ \hline
-&+&&&&&-\\ \hline
+&-&+&-&&&+\\ \hline\hline
&+&-&+&-&+&\end{array}\]
does not result from $A$ by replacing nonzeros with zeros.

In addition to the number of $(u,u'|v,v')$-ASMs, it is also of interest to determine the minimum and maximum number of nonzeros in an $(u,u'|v,v')$-ASM, For classical ASMs, these numbers are, respectively, $n$ (the permutation matrices) and $n^2/2$ ($n$ even) and $(n^2+1)/2$ ($n$ odd)
(the so-called diamond ASMs). If $u=u'=v=v'$  with the $+1$s and $-1$s alternating, then there is a  $(u,u'|v,v')$-ASM with $n$ nonzeros in a permutation set of places and $n$ is clearly the minimum number of nonzeros; the maximum number is $n^2$ if $n$ is odd and $n^2-n$ if $n$ is even.

Finally, we make the following observation.
We might allow the vectors $u,u',v, v'$ to have some zero coordinates (so no restriction on the signs of the first and/or last nonzero entries in 
certain rows and columns). To determine whether or not a $(u,u'|v,v')$-ASM exists under these circumstances, one need only  determine how to change the 0s to $\pm1$ so that conditions
(\ref{eq:nec1}), (\ref{eq:nec2}), and (\ref{eq:nec3}) hold. But this seems messy to formulate and thus not very useful.

\end{document}